\documentclass[11pt]{article}
\usepackage{amsmath}
\usepackage{amssymb}
\usepackage{amsthm}
\usepackage[usenames]{color}
\usepackage{amscd}
\usepackage{indentfirst}
\usepackage[colorlinks=true,linkcolor=webgreen,filecolor=webbrown,
citecolor=webgreen]{hyperref}
\definecolor{webgreen}{rgb}{0,.5,0}
\definecolor{webbrown}{rgb}{.6,0,0}
\usepackage{chngcntr}

\def\1{{\bf 1}}

\def\N{{\Bbb N}}
\def\Z{{\Bbb Z}}

\def\C{{\Bbb C}}

\def\div{\,|\,}
\def\ddiv{\,||\,}

\def\Reg{\operatorname{Reg}}

\def\id{\operatorname{id}}

\def\Sreg{\operatorname{S[reg]}}
\def\Srelpr{\operatorname{S[relpr]}}
\def\Treg{\operatorname{T[reg]}}
\def\Trelpr{\operatorname{T[relpr]}}
\def\Preg{\operatorname{P[reg]}}
\def\Qreg{\operatorname{Q[reg]}}
\def\Qrelpr{\operatorname{Q[relpr]}}
\def\Rreg{\operatorname{R[reg]}}
\def\Rrelpr{\operatorname{R[relpr]}}
\def\Phireg{\operatorname{\Phi[reg]}}

\def\Phiregmod{\operatorname{\Phi[regmod]}}

\def\Uregmod{\operatorname{U[regmod]}}
\def\Vregmod{\operatorname{V[regmod]}}

\newtheorem{theorem}{Theorem}[section]

\newtheorem{proposition}[theorem]{Proposition}

\newtheorem{remark}[theorem]{Remark}

\begin{document}

\title{\bf Some remarks on regular integers modulo $n$}
\author{Br\u{a}du\c{t} Apostol and L\'aszl\'o T\'oth}
\date{}
\maketitle

\centerline{Filomat {\bf 29} (2015), no 4, 687--701}

\centerline{Available at \url{http://www.pmf.ni.ac.rs/filomat}}

\begin{abstract} An integer $k$ is called regular (mod $n$) if
there exists an integer $x$ such that $k^2x\equiv k$ (mod $n$).
This holds true if and only if $k$ possesses a weak order (mod $n$),
i.e., there is an integer $m\ge 1$ such that $k^{m+1} \equiv k$ (mod $n$).
Let $\varrho(n)$ denote the number of regular integers (mod $n$) in the set $\{1,2,\ldots,n\}$.
This is an analogue of Euler's $\phi$ function. We introduce the multidimensional
generalization of $\varrho$, which is the analogue of Jordan's function. We establish identities
for the power sums of regular integers (mod $n$) and for some other finite sums and products over regular integers (mod $n$),
involving the Bernoulli polynomials, the Gamma function and the cyclotomic polynomials, among others.
We also deduce an analogue of Menon's identity and investigate the maximal orders of certain related functions.
\end{abstract}

{\sl 2010 Mathematics Subject Classification}: 11A25, 11B68, 11N37,
33B10, 33B15

{\sl Key Words and Phrases}: regular integer (mod $n$), Euler's
totient function, Jordan's function, Ramanujan's sum, unitary
divisor, Bernoulli numbers and polynomials, Gamma function, finite trigonometric
sums and products, cyclotomic polynomial

\section{Introduction}

Throughout the paper we use the notations: $\N:=\{1,2,\ldots\}$,
$\N_0 :=\{0,1,2,\ldots\}$, $\Z$ is the set of integers, ${\lfloor x
\rfloor}$ is the integer part of $x$, $\1$ is the function given by
$\1(n)=1$ ($n\in \N$), $\id$ is the function given by $\id(n)=n$
($n\in \N$), $\phi$ is Euler's totient function, $\tau(n)$ is the
number of divisors of $n$, $\mu$ is the M\"obius function,
$\omega(n)$ stands for the number of prime factors of $n$, $\Lambda$
is the von Mangoldt function, $\kappa(n):=\prod_{p\div n} p$ is the
largest squarefree divisor of $n$, $c_n(t)$ are the Ramanujan sums
defined by $c_n(t) :=\sum_{1\le k \le n, \gcd(k,n)=1} \exp(2\pi
ikt/n)$ ($n\in \N, t\in \Z$), $\zeta$ is the Riemann zeta function.
Other notations will be fixed inside the paper.

Let $n\in \N$ and $k\in \Z$. Then $k$ is called regular (mod $n$) if
there exists $x\in \Z$ such that $k^2x\equiv k$ (mod $n$). This
holds true if and only if $k$ possesses a weak order (mod $n$),
i.e., there is $m\in \N$ such that $k^{m+1} \equiv k$ (mod $n$).
Every $k\in \Z$ is regular (mod $1$). If $n>1$ and its prime power
factorization is $n=p_1^{\nu_1}\cdots p_r^{\nu_r}$, then $k$ is
regular (mod $n$) if and only if for every $i\in \{1,\ldots,r\}$
either $p_i\nmid k$ or $p_i^{\nu_i}\div k$. Also, $k$ is regular
(mod $n$) if and only if $\gcd (k,n)$ is a unitary divisor of $n$.
We recall that $d$ is said to be a unitary divisor of $n$ if $d\div
n$ and $\gcd (d,n/d)=1$, notation $d \ddiv n$. Note that if $n$ is
squarefree, then every $k\in \Z$ is regular (mod $n$). See the
papers \cite{AlkOsb2008, Mor1972,Mor1974,Tot2008} for further
discussion and properties of regular integers (mod $n$), and their
connection with the notion of regular elements of a ring in the
sense of J. von Neumann.

An integer $k$ is regular (mod $n$) if and only if $k+n$ is regular
(mod $n$). Therefore, it is justified to consider the set
\begin{equation*}
\Reg_n:=\{k\in \N: 1\le k\le n, k \text{ is regular (mod $n$)}\}
\end{equation*}
and the quantity $\varrho(n):=\# \Reg_n$. For example,
$\Reg_{12}=\{1,3,4,5,7,8,9,11,12\}$ and $\varrho(12)=9$. If $n$ is
squarefree, then $\Reg_n=\{1,2,\ldots,n\}$ and $\varrho(n)=n$. Note
that $1,n\in \Reg_n$ for every $n\in \N$. The arithmetic function
$\varrho$ is an analogue of Euler's $\phi$ function, it is
multiplicative and $\varrho(p^{\nu})=\phi(p^{\nu})+1=
p^{\nu}-p^{\nu-1}+1$ for every prime power $p^{\nu}$ ($\nu \in \N$).
Consequently,
\begin{equation} \label{varrho}
\varrho(n)= \sum_{d \ddiv n} \phi(d) \quad (n\in \N).
\end{equation}

See, e.g., \cite{McC1986} for general properties of unitary
divisors, in particular the unitary convolution of the arithmetic
functions $f$ and $g$ defined by $(f \times g)(n)=\sum_{d\ddiv n}
f(d) g(n/d)$. Here $f\times g$ preserves the multiplicativity
of the functions $f$ and $g$.  We refer to
\cite{Tot2008} for asymptotic properties of the function $\varrho$.

The function
\begin{equation*}
\overline{c}_n(t) :=\sum_{k\in \Reg_n} \exp(2\pi ikt/n) \quad (n\in
\N, t\in \Z),
\end{equation*}
representing an analogue of the Ramanujan sum $c_n(t)$ was
investigated in the paper \cite{HauTot2012}. We have
\begin{equation*}
\overline{c}_n(t)= \sum_{d \ddiv n} c_d(t) \quad (n\in \N, t\in \Z).
\end{equation*}

It turns out that for every fixed $t$ the function $n\mapsto
\overline{c}_n(t)$ is multiplicative, $\overline{c}_n(0)=\varrho(n)$
and $\overline{c}_n(1)=\overline{\mu}(n)$ is the characteristic
function of the squarefull integers $n$.

The gcd-sum function is defined by $P(n) :=\sum_{k=1}^n \gcd(k,n)
=\sum_{d\div n} d \, \phi(n/d)$, see \cite{Tot2010}. The following
analogue of the gcd-sum function was introduced in the paper
\cite{Tot2009}:
\begin{equation*}
\widetilde{P}(n):= \sum_{k\in \Reg_n} \gcd(k,n).
\end{equation*}

One has
\begin{equation*}
\widetilde{P}(n)=\sum_{d\ddiv n} d\, \phi(n/d)= n \prod_{p\div n}
\left(2-\frac1{p} \right) \quad (n\in \N),
\end{equation*}
the asymptotic properties of $\widetilde{P}(n)$ being investigated
in \cite{KonKat2010,Tot2010, ZhaZha2010,ZhaZha2011}.

In the present paper we discuss some further properties of the
regular integers (mod $n$). We first introduce the multidimensional
generalization $\varrho_r$ ($r\in \N$) of the function $\varrho$,
which is the analogue of the Jordan function $\phi_r$, where
$\phi_r(n)$ is defined as the number of ordered $r$-tuples
$(k_1,\ldots,k_r)\in \{1,\ldots,n\}^r$ such that
$\gcd(k_1,\ldots,k_r)$ is prime to $n$ (see, e.g., \cite{McC1986,
Sch1999}). Then we consider the sum $\Sreg_r(n)$ of $r$-th powers of
the regular integers (mod $n$) belonging to $\Reg_n$. In the case
$r\in \N$ we deduce an exact formula for $\Sreg_r(n)$ involving the
Bernoulli numbers $B_m$. For a positive real number $r$ we derive an
asymptotic formula for $\Sreg_r(n)$. We combine the functions
$\overline{c}_n(t)$ and $\widetilde{P}(n)$ defined above and
establish identities for sums, respectively products over the
integers in $\Reg_n$ concerning the Bernoulli polynomials $B_m(x)$,
the Gamma function $\Gamma$, the cyclotomic polynomials $\Phi_m(x)$
and certain trigonometric functions. We point out that for $n$
squarefree these identities reduce to the corresponding ones over
$\{1,2,\ldots,n\}$. We also deduce an analogue of Menon's identity
and investigate the maximal orders of some related functions.

\newpage

\section{A generalization of the function $\varrho$}

For $r\in \N$ let $\varrho_r(n)$ be the number of ordered $r$-tuples
$(k_1,\ldots,k_r)\in \{1,\ldots,n\}^r$ such that
$\gcd(k_1,\ldots,k_r)$ is regular (mod $n$). If $r=1$, then
$\varrho_1=\varrho$. The arithmetic function $\varrho_r$ is the
analogue of the Jordan function $\phi_r$, defined in the
Introduction and verifying $\phi_r(n)=n^r\prod_{p\div n} (1-1/p^r)$
($n\in \N$).

\begin{proposition} i) For every $r, n\in \N$,
\begin{equation*}
\varrho_r(n)= \sum_{d \ddiv n} \phi_r(d).
\end{equation*}

ii) The function $\varrho_r$ is multiplicative and for every prime power $p^{\nu}$ ($\nu \in \N$),
\begin{equation*}
\varrho_r(p^{\nu}) = p^{r\nu} -p^{r(\nu-1)}+1.
\end{equation*}
\end{proposition}

\begin{proof} i) The integer $\gcd(k_1,\ldots,k_r)$ is regular (mod $n$)
if and only if $\gcd(\gcd(k_1,\ldots,k_r),n)\ddiv n$, that is
$\gcd(k_1,\ldots,k_r,n)\ddiv n$ and grouping the $r$-tuples
$(k_1,\ldots,k_r)$ according to the values
$\gcd(k_1,\ldots,k_r,n)=d$ we deduce that
\begin{equation*}
\varrho_r(n)= \sum_{\substack{(k_1,\ldots,k_r)\in \{1,\ldots,n\}^r\\
\gcd(k_1,\ldots,k_r) \text{ regular (mod $n$)}}} 1
= \sum_{d \ddiv n} \sum_{\substack{(k_1,\ldots,k_r)\in \{1,\ldots,n\}^r\\
\gcd(k_1,\ldots,k_r,n)=d}} 1
\end{equation*}
\begin{equation*}
= \sum_{d \ddiv n} \sum_{\substack{(\ell_1,\ldots,\ell_r)\in \{1,\ldots,n/d\}^r\\
\gcd(\ell_1,\ldots,\ell_r,n/d)=1}} 1,
\end{equation*}
where the inner sum is $\phi_r(n/d)$, according to its definition.

ii) Follows at once by i).
\end{proof}

More generally, for a fixed real number $s$ let
$\phi_s(n)=\sum_{d\div n} d^s \mu(n/d)$ be the generalized Jordan
function and define $\varrho_s$ by
\begin{equation} \label{varrho_s}
\varrho_s(n)=\sum_{d\ddiv n} \phi_s(d) \quad (n\in \N).
\end{equation}

The functions $\phi_s$ and $\varrho_s$ (which will be used in the
next results of the paper) are multiplicative and for every prime
power $p^{\nu}$ ($\nu \in \N$) one has $\phi_s(p^{\nu})= p^{s\nu}
-p^{s(\nu-1)}$ and $\varrho_s(p^{\nu})= p^{s\nu} -p^{s(\nu-1)}+1$.
Note that $\phi_{-s}(n) = n^{-s} \prod_{p^{\nu}\ddiv n} \left(1-p^s
\right)$ and $\varrho_{-s}(n) = n^{-s} \prod_{p^{\nu}\ddiv n}
\left(p^{s\nu} -p^s+1\right)$.

\begin{proposition} If $s>1$ is a real number, then
\begin{equation} \label{asympt_varrho_r}
\sum_{n\le x} \varrho_s(n)= \frac{x^{s+1}}{s+1} \prod_p
\left(1-\frac1{p^{s+1}}+ \frac{p-1}{p(p^{s+1}-1)} \right) + O(x^s).
\end{equation}
\end{proposition}

\begin{proof}
We need the following asymptotics. Let $s>0$ be fixed real number.
Then uniformly for real $x>1$ and $t\in \N$,
\begin{equation} \label{Jordan_t}
\phi_s(x,t):= \sum_{\substack{n\le x\\ \gcd(n,t)=1}} \phi_s(n)=
\frac{x^{s+1}}{(s+1)\zeta(s+1)}\cdot
\frac{t^s\phi(t)}{\phi_{s+1}(t)} +O(x^s 2^{\omega(t)}).
\end{equation}

To show \eqref{Jordan_t} use the known estimate, valid for every fixed $s>0$ and $t\in \N$,
\begin{equation} \label{asympt_relprime}
\sum_{\substack{n\le x \\ \gcd(n,t)=1}} n^s = \frac{x^{s+1}}{s+1}\cdot
\frac{\phi(t)}{t}+ O\left(x^s 2^{\omega(t)} \right).
\end{equation}

We obtain
\begin{equation*}
\phi_s(x,t)= \sum_{\substack{de=n\le x \\ \gcd(n,t)=1}} \mu(d)e^s =
\sum_{\substack{d\le x \\ \gcd(d,t)=1}} \mu(d) \sum_{\substack{e\le x/d
\\ \gcd(e,t)=1}}e^s
\end{equation*}
\begin{equation*}
=\sum_{\substack{d\le x\\ \gcd(d,t)=1}} \mu(d) \left(
\frac{(x/d)^{s+1}}{s+1} \cdot \frac{\phi(t)}{t} + O \left(
(x/d)^s 2^{\omega(t)}\right)\right)
\end{equation*}
\begin{equation*}
=\frac{x^{s+1}}{s+1}\cdot \frac{\phi(t)}{t}\sum_{\substack{d=1\\
\gcd(d,t)=1}}^{\infty} \frac{\mu(d)}{d^{s+1}} + O \left( x^{s+1}
\sum_{d>x} \frac1{d^{s+1}}\right) + O\left(x^s 2^{\omega(t)}\right),
\end{equation*}
giving \eqref{Jordan_t}. Now  from \eqref{varrho_s} and
\eqref{Jordan_t},
\begin{equation*}
\sum_{n\le x} \varrho_s(n) = \sum_{\substack{de= n\le x\\
\gcd(d,e)=1}} \phi_s(e) = \sum_{d\le x} \sum_{\substack{e\le x/d\\
\gcd(e,d)=1}} \phi_s(e) = \sum_{d\le x} \phi_s(x/d,d)
\end{equation*}
\begin{equation*}
=\frac{x^{s+1}}{(s+1)\zeta(s+1)} \sum_{d=1} ^{\infty}
\frac{\phi(d)}{d\phi_{s+1}(d)} + O \left( x^{s+1} \sum_{d>x}
\frac{\phi(d)}{d\phi_{s+1}(d)}\right) + O\left(x^s \sum_{d\le x}
\frac{2^{\omega(d)}}{d^s} \right),
\end{equation*}
and for $s>1$ this leads to \eqref{asympt_varrho_r}.
\end{proof}

Compare \eqref{asympt_varrho_r} to the corresponding formula for the
Jordan function $\phi_s$, i.e., to \eqref{Jordan_t} with $t=1$.

\begin{remark} {\rm For the function $\varrho$ one has
\begin{equation*}
\sum_{n\le x} \varrho(n) = \frac1{2} \prod_p
\left(1-\frac1{p^2(p+1)} \right)x^2 +R(x),
\end{equation*}
where $R(x)=O(x\log^3 x)$ can be obtained by the elementary
arguments given above. This can be improved into $R(x)=O(x\log x)$
using analytic methods. See \cite{Tot2008} for references.}
\end{remark}


\section{A general scheme}

In order to give exact formulas for certain sums and products over
the regular integers (mod $n$) we first present a simple result for
a general sum over $\Reg_n$, involving a weight function $w$ and an
arithmetic function $f$. It would be possible to consider a more
general sum, namely over the ordered $r$-tuples $(k_1,\ldots,k_r)\in
\{1,\ldots,n\}^r$ such that $\gcd(k_1,\ldots,k_r)$ is regular (mod
$n$), but we confine ourselves to the following result. See
\cite{Tot2011w} for another similar scheme concerning weighted
gcd-sum functions.

\begin{proposition} \label{prop_general} i) Let $w:\N^2 \to \C$ and $f:\N \to \C$ be
arbitrary functions and consider the sum
\begin{equation*}
R_{w,f}(n):= \sum_{k\in \Reg_n} w(k,n) f(\gcd(k,n)).
\end{equation*}

Then
\begin{equation} \label{R_repr}
R_{w,f}(n)= \sum_{d\ddiv n} f(d) \sum_{\substack{j=1\\
\gcd(j,n/d)=1}}^{n/d} w(dj,n) \quad (n\in \N).
\end{equation}

ii) Assume that there is a function $g:(0,1] \to \C$ such that
$w(k,n)=g(k/n)$ ($1\le k\le n$) and let
\begin{equation*}
\overline{G}(n)=\sum_{\substack{k=1\\ \gcd(k,n)=1}}^n g(k/n) \quad
(n\in \N).
\end{equation*}

Then
\begin{equation} \label{R_repr_g}
R_{w,f}(n)=\sum_{d\ddiv n} f(d) \overline{G}(n/d) \quad (n\in \N).
\end{equation}
\end{proposition}

\begin{proof} i) Using that $k$ is regular (mod $n$) if and only if
$\gcd(k,n)\ddiv n$ and grouping the terms according to the values of
$\gcd(k,n)=d$ and denoting $k=dj$ we have at once
\begin{equation*}
R_{w,f}(n)= \sum_{d\ddiv n} f(d) \sum_{\substack{k=1\\
\gcd(k,n)=d}}^n w(k,n) = \sum_{d\ddiv n} f(d) \sum_{\substack{j=1\\
\gcd(j,n/d)=1}}^{n/d} w(dj,n).
\end{equation*}

ii) Now
\begin{equation*}
\sum_{j=1}^{n/d} w(dj,n)= \sum_{j=1}^{n/d} g(j/(n/d))=
\overline{G}(n/d).
\end{equation*}
\end{proof}

\begin{remark} \label{Remark_relprime} {\rm For the function $g$ given above let
\begin{equation*}
G(n):=\sum_{k=1}^n g(k/n).
\end{equation*}

Then we have
\begin{equation} \label{connect_G}
\overline{G}(n)= \sum_{d\div n} \mu(d) G(n/d) \quad (n\in \N).
\end{equation}

Indeed, as it is well known, $\overline{G}(n)=\sum_{k=1}^n g(k/n)
\sum_{d\mid \gcd(k,n)} \mu(d)$, giving \eqref{connect_G}. }
\end{remark}


\section{Power sums of regular integers (mod $n$)}
\label{Sect_Power_sums}

In this section we investigate the sum of $r$-th powers ($r\in \N$)
of the regular integers (mod $n$). Let $B_m$ ($m\in \N_0$) be the
Bernoulli numbers defined by the exponential generating function
\begin{equation*}
\frac{t}{e^t-1}=\sum_{m=0}^{\infty} B_m \frac{t^m}{m!}.
\end{equation*}

Here $B_0=1$, $B_1=-1/2$, $B_2=1/6$, $B_4=-1/30$, $B_m=0$ for every
$m\ge 3$, $m$ odd and one has the recurrence relation
\begin{equation} \label{Bernoulli_rec}
B_m = \sum_{j=0}^m \binom{m}{j} B_j \quad (m\ge 2).
\end{equation}

It is well known that for every $n,r\in \N$,
\begin{equation*}
S_r(n):= \sum_{k=1}^n k^r = \frac1 {r+1} \sum_{m=0}^r (-1)^m
\binom{r+1}{m} B_m n^{r+1-m}
\end{equation*}
\begin{equation} \label{sum_r}
=\frac{n^r}{2} + \frac1 {r+1} \sum_{m=0}^{\lfloor r/2 \rfloor}
\binom{r+1}{2m} B_{2m} n^{r+1-2m}.
\end{equation}

From here one obtains, using the same device as that given in Remark
\ref{Remark_relprime} that for every $n,r\in \N$ with $n\ge 2$,
\begin{equation} \label{S_relprime}
\Srelpr_r(n):= \sum_{\substack{k=1\\ \gcd(k,n)=1}}^n k^r =
\frac{n^r}{r+1} \sum_{m=0}^{\lfloor r/2 \rfloor} \binom{r+1}{2m}
B_{2m}\phi_{1-2m}(n),
\end{equation}
where $\phi_{1-2m}(n)= n^{1-2m} \prod_{p \div n} \left(1-p^{2m-1}
\right)$. Formula \eqref{S_relprime} was given in \cite{Sin2009}.
Here we prove the following result.

\begin{proposition} \label{prop_powersum}
For every $n,r\in \N$,
\begin{equation} \label{sum_reg_r}
\Sreg_r(n) := \sum_{k\in \Reg_n} k^r =\frac{n^r}{2}+ \frac{n^r}{r+1}
\sum_{m=0}^{\lfloor r/2 \rfloor} \binom{r+1}{2m} B_{2m}\,
\varrho_{1-2m}(n),
\end{equation}
where
\begin{equation*}
\varrho_{1-2m}(n) = n^{1-2m} \prod_{p^\nu \ddiv n} \left(
p^{(2m-1)\nu}- p^{2m-1} +1 \right)
\end{equation*}
is the generalized $\varrho$ function, discussed in Section 2.
\end{proposition}

\begin{proof} Applying \eqref{R_repr} for $w(k,n)=k^r$ and $f=\1$ we
have
\begin{equation*}
\Sreg_r(n) = \sum_{d\ddiv n} \sum_{\substack{j=1\\
\gcd(j,n/d)=1}}^{n/d} (dj)^r = \sum_{d\ddiv n} d^r \Srelpr_r(n/d).
\end{equation*}

Now by \eqref{S_relprime} we deduce
\begin{equation*}
\Sreg_r(n) = n^r+ \sum_{\substack{d\ddiv n\\d<n}} d^r \left(
\frac{(n/d)^r}{r+1} \sum_{m=0}^{\lfloor r/2 \rfloor} \binom{r+1}{2m}
B_{2m} \phi_{1-2m}(n/d) \right)
\end{equation*}
\begin{equation*}
= n^r+ \frac{n^r}{r+1} \sum_{m=0}^{\lfloor r/2 \rfloor}
\binom{r+1}{2m} B_{2m} \sum_{\substack{d\ddiv n\\d<n}}
\phi_{1-2m}(n/d)
\end{equation*}
\begin{equation*}
= n^r+ \frac{n^r}{r+1} \sum_{m=0}^{\lfloor r/2 \rfloor}
\binom{r+1}{2m} B_{2m} \sum_{\substack{d\ddiv n\\d>1}}
\phi_{1-2m}(d)
\end{equation*}
\begin{equation*}
= n^r - \frac{n^r}{r+1} \sum_{m=0}^{\lfloor r/2 \rfloor}
\binom{r+1}{2m} B_{2m} + \frac{n^r}{r+1} \sum_{m=0}^{\lfloor r/2
\rfloor} \binom{r+1}{2m} B_{2m} \sum_{d\ddiv n} \phi_{1-2m}(d).
\end{equation*}

Here $\sum_{d\ddiv n} \phi_{1-2m}(n) = \varrho_{1-2m}(d)$ by
\eqref{varrho_s}. Also, by \eqref{Bernoulli_rec},
\begin{equation*}
\sum_{m=0}^{\lfloor r/2 \rfloor} \binom{r+1}{2m} B_{2m}
=\frac{r+1}{2}
\end{equation*}
and this completes the proof.
\end{proof}

For example, in the cases $r=1,2,3,4$ we deduce that for every $n\in
\N$,
\begin{equation} \label{sum_1}
\Sreg_1(n)= \frac{n(\varrho(n)+1)}{2},
\end{equation}
\begin{equation} \label{sum_2}
\Sreg_2(n) = \frac{n^2}{2} + \frac{n^2\varrho(n)}{3}+ \frac{n}{6}
\prod_{p^{\nu} \ddiv n} \left(p^{\nu}- p+1\right),
\end{equation}
\begin{equation*}
\Sreg_3(n) = \frac{n^3}{2} + \frac{n^3\varrho(n)}{4}+ \frac{n^2}{4}
\prod_{p^{\nu} \ddiv n} \left(p^{\nu}- p+1\right),
\end{equation*}
\begin{equation*}
\Sreg_4(n) = \frac{n^4}{2} + \frac{n^4\varrho(n)}{5}+ \frac{n^3}{3}
\prod_{p^{\nu} \ddiv n} \left(p^{\nu}- p+1\right) - \frac{n}{30}
\prod_{p^{\nu} \ddiv n} \left(p^{3\nu}- p^3+1\right).
\end{equation*}

The formula \eqref{sum_1} was obtained in \cite[Th.\ 3]{Tot2008} and
\cite[Sec.\ 2]{ApoPet2012}, while \eqref{sum_2} was given in a
different form in \cite[Prop.\ 1]{ApoPet2012}. Note that if $n$ is
squarefree, then \eqref{sum_reg_r} reduces to \eqref{sum_r}.

For a real number $s$ consider now the slightly more general sum
\begin{equation*}
\Sreg_s(n,x): = \sum_{\substack{k\le x \\ k \text{\rm \ regular (mod
$n$)}}} k^s.
\end{equation*}

\begin{proposition} \label{asympt_power}
Let $s\ge 0$ be a fixed real number. Then uniformly for real $x>1$
and $n\in \N$,
\begin{equation*}
\Sreg_s(n,x)= \frac{x^{s+1}}{s+1}\cdot \frac{\varrho(n)}{n} +
O\left(x^s 3^{\omega(n)} \right).
\end{equation*}
\end{proposition}

\begin{proof} Similar to the proof of Proposition \ref{prop_general},
\begin{equation*}
\Sreg_s(n,x)= \sum_{\substack{k\le x\\ \gcd(k,n)\ddiv
n}} k^s = \sum_{d\ddiv n} d^s \sum_{\substack{j\le x/d\\
\gcd(j,n/d)=1}} j^s.
\end{equation*}

Now using the estimate \eqref{asympt_relprime} we deduce
\begin{equation*}
\Sreg_s(n,x)=  \sum_{d\ddiv n} d^s \left(\frac{(x/d)^{s+1}
\phi(n/d)}{(s+1)(n/d)}+ O\left((x/d)^s 2^{\omega(n/d)} \right)
\right)
\end{equation*}
\begin{equation*}
= \frac{x^{s+1}}{(s+1)n} \sum_{d\ddiv n} \phi(n/d)+ O \left(x^s
\sum_{d\ddiv n} 2^{\omega(n/d)} \right),
\end{equation*}
and using \eqref{varrho} the proof is complete.
\end{proof}


\section{Identities for other sums and products over regular integers (mod $n$)}

\subsection{\bf Sums involving Bernoulli polynomials}

Let $B_m(x)$ ($m\in \N_0$) be the Bernoulli polynomials defined by
\begin{equation*}
\frac{te^{xt}}{e^t-1}= \sum_{m=0}^{\infty} B_m(x) \frac{t^m}{m!}.
\end{equation*}

Here $B_0(x)=1$, $B_1(x)=x-1/2$, $B_2(x)=x^2-x+1/6$, $B_3(x)=x^3-3x^2/2+ x/2$,
$B_m(0)=B_m$ ($m\in \N_0$) are the Bernoulli numbers already defined
in Section \ref{Sect_Power_sums} and one has the recurrence relation
\begin{equation*}
B_m(x) = \sum_{j=0}^m \binom{m}{j} B_j x^{m-j} \quad (m\in \N_0).
\end{equation*}

It is well known (see, e.g., \cite[Sect.\ 9.1]{Coh2007}) that for
every $n,m\in \N$, $m\ge 2$,
\begin{equation} \label{sum_Bernoulli}
T_m(n):= \sum_{k=1}^n B_m(k/n) = \frac{B_m}{n^{m-1}}.
\end{equation}

Furthermore, applying  \eqref{connect_G} one obtains from
\eqref{sum_Bernoulli} that for every $n,m\in \N$, $m\ge 2$,
\begin{equation} \label{sum_Bernoulli_relprime}
\Trelpr_m(n) := \sum_{\substack{k=1\\ \gcd(k,n)=1}}^n B_m(k/n) = B_m
\phi_{1-m}(n),
\end{equation}
where $\phi_{1-m}(n) =n^{1-m} \prod_{p\div n}
\left(1-p^{m-1}\right)$. See \cite[Sect.\ 9.9, Ex. 7]{Coh2007}. We
now show the validity of the next formula:\\[2 pt]

\noindent {\bf Proposition 5.1.1} \label{prop_Bernoulli_reg}
\textit{For every $n,m\in \N$, $m\ge 2$,
\begin{equation} \label{sum_Bernoulli_reg}
\Treg_m(n) := \sum_{k\in \Reg_n} B_m(k/n) = B_m \varrho_{1-m}(n),
\end{equation}
where $\varrho_{1-m}(n)=n^{1-m} \prod_{p^{\nu} \ddiv  n}
\left(p^{(m-1)\nu}-p^{m-1}+1 \right)$.}

\begin{proof} Choosing $g(x)=B_m(x)$ and $f=\1$ we deduce from \eqref{R_repr_g}
by using \eqref{sum_Bernoulli_relprime} that
\begin{equation*}
\Treg_m(n) = \sum_{d\ddiv n} \Trelpr_m(d)
\end{equation*}
\begin{equation*}
= B_m \sum_{d\ddiv n} \phi_{1-m}(d)= B_m \varrho_{1-m}(n),
\end{equation*}
according to \eqref{varrho_s}.
\end{proof}

\noindent{\bf Remark 5.1.2}
{\rm In the case $m=1$ a direct computation and \eqref{sum_1}
show that $\Treg_1(n)=1/2$. Also, \eqref{sum_Bernoulli_reg} can be
put in the form
\begin{equation*}
\sum_{\substack{k=0\\ k \text{\rm \ regular (mod $n$)}}}^{n-1}
B_m(k/n) = B_m \varrho_{1-m}(n),
\end{equation*}
which holds true for every $n,m\in \N$, also for $m=1$. }


\subsection{\bf Sums involving gcd's and the $\exp$ function}

Consider in what follows the function
\begin{equation*}
\Preg_{f,t}(n) := \sum_{k\in \Reg_n} f(\gcd(k,n)) \exp(2\pi i kt/n)
\quad (n\in \N, t\in \Z),
\end{equation*}
where $f$ is an arbitrary arithmetic function. For $t=0$ and
$f(n)=n$ ($n\in \N$) we reobtain the function $\widetilde{P}(n)$ and
for $f=\1$ we have $\overline{c}_n(t)$, the analogue of the
Ramanujan sums, both given in the Introduction. We have\\[2pt]

\noindent {\bf Proposition 5.2.1}
\textit{For every $f$ and every $n\in \N$ and $t\in \Z$,}
\begin{equation*}
\Preg_{f,t}(n) = \sum_{d \ddiv n} f(d) c_{n/d}(t).
\end{equation*}

\textit{If $f$ is integer valued and multiplicative (in particular, if
$f=\id$), then $n\mapsto \Preg_{f,t}(n)$ also has these properties.}

\begin{proof} Choosing $g(x)=\exp(2\pi itx)$ from \eqref{R_repr_g} we
deduce at once that
\begin{equation*}
\Preg_{f,t}(n)= \sum_{d\ddiv n} f(d) \sum_{\substack{j=1\\
\gcd(j,n/d)=1}}^{n/d} \exp(2\pi ijt/(n/d)) =\sum_{d \ddiv n} f(d)
c_{n/d}(t).
\end{equation*}
\end{proof}

For $t=1$ and $f=\id$ this gives the multiplicative function
\begin{equation*}
\Preg_{\id,1}(n) = \sum_{d\ddiv n} d\mu(n/d),
\end{equation*}
not investigated in the literature, as far as we know. Here
$\Preg_{\id,1}(p^{\nu})= p-1$ for every prime $p$ and
$\Preg_{\id,1}(p^{\nu})= p^{\nu}$ for every prime power $p^{\nu}$
with $\nu \ge 2$.\\[2 pt]

\noindent {\bf Proposition 5.2.2}\label{asymp_DFT}
\textit{We have}
\begin{equation*}
\sum_{n\le x} \Preg_{\id,1}(n)= \frac{x^2}{2}\prod_p
\left(1-\frac1{p^2}+\frac1{p^3} \right) + O(x\log^2 x).
\end{equation*}

\begin{proof} Using \eqref{asympt_relprime} for $s=1$ we deduce
\begin{equation*}
\sum_{n\le x} \Preg_{\id,1}(n)= \sum_{d\le x} \mu(d)
\sum_{\substack{\delta \le x/d\\ \gcd(\delta,d)=1}} \delta
\end{equation*}
\begin{equation*}
= \sum_{d\le x} \mu(d) \left(\frac{\phi(d)(x/d)^2}{2d} +
O((x/d)2^{\omega(d)}) \right)
\end{equation*}
\begin{equation*}
= \frac{x^2}{2} \sum_{d=1}^{\infty} \frac{\mu(d)\phi(d)}{d^3}+
O\left(x^2 \sum_{d>x} \frac1{d^2}\right)+  O\left(x\sum_{d\le x}
\frac{2^{\omega(d)}}{d}\right),
\end{equation*}
giving the result.
\end{proof}


\subsection{\bf An analogue of Menon's identity}

Our next result is the analogue of Menon's identity (\cite{Men1965},
see also \cite{Tot2011})
\begin{equation} \label{Menon_id}
\sum_{\substack{k=1\\ \gcd(k,n)=1}}^n \gcd(k-1,n) = \phi(n)\tau(n)
\quad (n\in \N).
\end{equation}

\noindent {\bf Proposition 5.3.1} \label{prop_analog_Menon}
\textit{For every $n\in \N$,}
\begin{equation*}
\sum_{k \in \Reg_n} \gcd(k-1,n) = \sum_{d\ddiv n} \phi(d)\tau(d)=
\prod_{p^\nu \ddiv n} \left( p^{\nu-1} (p-1)(\nu+1)+1 \right).
\end{equation*}

\begin{proof} Applying \eqref{R_repr} for $w(k,n)=\gcd(k-1,n)$ and $f=\1$ we
deduce
\begin{equation*}
S_n:= \sum_{k\in \Reg_n} \gcd(k-1,n) =\sum_{d\ddiv n} \sum_{\substack{j=1\\
\gcd(j,n/d)=1}}^{n/d} \gcd(dj-1,n)
\end{equation*}
\begin{equation*}
= \sum_{d\ddiv n} \sum_{\substack{j=1\\
\gcd(j,n/d)=1}}^{n/d} \gcd(dj-1,n/d),
\end{equation*}
since $\gcd(dj-1,d)=1$ for every $d$ and $j$. Now we use the
identity
\begin{equation*}
\sum_{\substack{k=1\\ \gcd(k,n)=1}}^n \gcd(ak-1,n) = \phi(n)\tau(n)
\quad (n\in \N),
\end{equation*}
valid for every fixed $a\in \N$ with $\gcd(a,n)=1$, see \cite[Cor.\
14]{Tot2011} (for $a=1$ this reduces to \eqref{Menon_id}). Choose
$a=d$. Since $d\ddiv n$ we have $\gcd(d,n/d)=1$ and obtain
\begin{equation*}
S_n= \sum_{d\ddiv n} \phi(n/d)\tau(n/d) = \sum_{d\ddiv n}
\phi(d)\tau(d).
\end{equation*}
\end{proof}


\subsection{\bf Trigonometric sums}

Further identities for sums over $\Reg_n$ can be derived. As
examples, consider the following known trigonometric identities. For
every $n\in \N$, $n\ge 2$,
\begin{equation*}
\sum_{k=1}^n \cos^2 \left(\frac{k\pi}{n}\right) = \frac{n}{2};
\end{equation*}
furthermore, for every $n\in \N$ odd number,
\begin{equation*}
\sum_{k=1}^n \tan^2 \left(\frac{k\pi}{n}\right) = n^2-n;
\end{equation*}
and also for every $n\in \N$ odd,
\begin{equation*}
\sum_{k=1}^n \tan^4 \left(\frac{k\pi}{n}\right) =
\frac1{3}(n^4-4n^2+3n).
\end{equation*}

See, for example, \cite{BecHal2010} for a discussion and proofs of
these identities. See \cite[Appendix 3]{NZM} for other similar identities.
By the approach given in Remark
\ref{Remark_relprime} we deduce that for every $n\in \N$,
\begin{equation*}
\sum_{\substack{k=1\\ \gcd(k,n)=1}}^n \cos^2
\left(\frac{k\pi}{n}\right) = \frac{\phi(n)+\mu(n)}{2};
\end{equation*}
for every $n\in \N$ odd number,
\begin{equation*}
\sum_{\substack{k=1\\ \gcd(k,n)=1}}^n \tan^2
\left(\frac{k\pi}{n}\right) = \phi_2(n)-\phi(n);
\end{equation*}
and for every $n\in \N$ odd,
\begin{equation*}
\sum_{\substack{k=1\\ \gcd(k,n)=1}}^n  \tan^4
\left(\frac{k\pi}{n}\right) = \frac1{3}
(\phi_4(n)-4\phi_2(n)+3\phi(n)).
\end{equation*}

This gives the next results. The proof is similar to the proofs
given above.\\[2 pt]

\noindent {\bf Proposition 5.4.1}
\textit{For every $n\in \N$,}
\begin{equation*}
\sum_{k\in \Reg_n} \cos^2 \left(\frac{k\pi}{n}\right) =
\frac{\varrho(n)+\overline{\mu}(n)}{2},
\end{equation*}
\textit{where $\overline{\mu}(n)=\sum_{d\ddiv  n} \mu(d)$ is the
characteristic function of the squarefull integers $n$, given in the
Introduction.}\\[2 pt]

\noindent {\bf Proposition 5.4.2}
\textit{For every $n\in \N$ odd number,}
\begin{equation*}
\sum_{k\in \Reg_n} \tan^2 \left(\frac{k\pi}{n}\right) = \varrho_2(n)
-\varrho(n),
\end{equation*}
\begin{equation*}
\sum_{k\in \Reg_n} \tan^4 \left(\frac{k\pi}{n}\right) = \frac1{3}
(\varrho_4(n)-4\varrho_2(n)+3\varrho(n)).
\end{equation*}


\subsection{\bf The product of numbers in $\Reg_n$}

It is known (see, e.g., \cite[p.\ 197, Ex.\ 24]{NZM}) that for every
$n\in \N$,
\begin{equation} \label{prod_rel_prime}
\Qrelpr(n):= \prod_{\substack{k=1\\
\gcd(k,n)=1}}^n k = n^{\phi(n)} A(n),
\end{equation}
where
\begin{equation*}
A(n) = \prod_{d\div n} (d!/d^d)^{\mu(n/d)}.
\end{equation*}

We show that\\[2 pt]

\noindent {\bf Proposition 5.5.1} \label{prop_product_reg_int}
\textit{For every $n\in \N$,}
\begin{equation*}
\Qreg(n):= \prod_{k \in \Reg_n} k = n^{\varrho(n)} \prod_{d\ddiv n}
A(d).
\end{equation*}

\begin{proof} Choosing $w(k,n)=\log k$ and $f=\1$ in Proposition \ref{prop_general} we have
\begin{equation*}
\log \Qreg(n) = \sum_{k\in \Reg_n} \log k = \sum_{d\ddiv n} \sum_{\substack{j=1\\
\gcd(j,n/d)=1}}^{n/d} \log(dj)
\end{equation*}
\begin{equation*}
= \sum_{d\ddiv n} \left(\phi(n/d)\log d + \log \Qrelpr(n/d) \right)
\end{equation*}
\begin{equation*}
= \sum_{d\ddiv n} \left(\phi(d)\log (n/d)  + \log \Qrelpr(d) \right)
\end{equation*}
\begin{equation*}
= (\log n) \sum_{d\ddiv n} \phi(d) - \sum_{d\ddiv n} \phi(d) \log d
+ \sum_{d\ddiv n} \log \Qrelpr(d).
\end{equation*}

Hence,
\begin{equation*}
\Qreg(n) = n^{\varrho(n)} \prod_{d\ddiv n}
\frac{\Qrelpr(d)}{d^{\phi(d)}}.
\end{equation*}

Now the result follows from the identity \eqref{prod_rel_prime}.
\end{proof}


\subsection{\bf Products involving the Gamma function}

Let $\Gamma$ be the Gamma function defined for $x>0$ by
\begin{equation*}
\Gamma(x)= \int_0^{\infty} e^{-t}t^{x-1}\, dt.
\end{equation*}

It is well known that for every $n\in \N$,
\begin{equation} \label{prod_Gamma}
R(n): = \prod_{k=1}^n \Gamma(k/n) =
\frac{(2\pi)^{(n-1)/2}}{\sqrt{n}},
\end{equation}
which is a consequence of Gauss' multiplication formula.
For the $q$-analogs of the Gamma and Beta functions and the multiplication formula see the recent papers
\cite{Ege, EgeYld} published in this journal.
Furthermore, for every $n\in \N$, $n\ge 2$,
\begin{equation} \label{prod_Gamma_rel_prime}
\Rrelpr(n):= \prod_{\substack{k=1\\ \gcd(k,n)=1}}^n \Gamma(k/n) =
\frac{(2\pi)^{\phi(n)/2}}{\exp (\Lambda(n)/2)},
\end{equation}
see \cite{SanTot1989,Mar2009}.\\[2 pt]

\noindent {\bf Proposition 5.6.1}\label{prop_Gamma_reg}
\textit{For every $n\in \N$,}
\begin{equation} \label{prod_reg_Gamma}
\Rreg(n):= \prod_{k\in \Reg_n} \Gamma(k/n) =
\frac{(2\pi)^{(\varrho(n)-1)/2}}{\sqrt{\kappa(n)}}.
\end{equation}

\begin{proof} Choosing $g=\log \Gamma$ and $f=\1$ in \eqref{R_repr_g} and
using \eqref{prod_Gamma_rel_prime} we deduce
\begin{equation*}
\log \Rreg(n) = \sum_{k\in \Reg_n} \log \Gamma(k/n) = \sum_{d\ddiv
n} \log \Rrelpr(d)
\end{equation*}
\begin{equation*}
=\sum_{\substack{d\ddiv n\\ d>1}} \left(\frac{\log 2\pi}{2} \phi(d)-
\frac1{2} \Lambda(d)\right)
\end{equation*}
\begin{equation*}
=\sum_{d\ddiv n} \left(\frac{\log 2\pi}{2} \phi(d)- \frac1{2}
\Lambda(d)\right)  - \frac{\log 2\pi}{2} = \frac{\log 2\pi}{2}
(\varrho(n)-1) - \frac1{2} \sum_{d\ddiv n} \Lambda(d),
\end{equation*}
where the last sum is $\log \kappa(n)$.
\end{proof}

For squarefree $n$ \eqref{prod_reg_Gamma} reduces to
\eqref{prod_Gamma}.


\subsection{\bf Identities involving cyclotomic polynomials}

Let $\Phi_n(x)$ ($n\in \N$) stand for the cyclotomic polynomials (see, e.g., \cite[Ch.\ 13]{IreRos1990})
defined by
\begin{equation*}
\Phi_n(x)= \prod_{\substack{k=1\\ \gcd(k,n)=1}}^n \left(x-\exp(2\pi
ik/n) \right).
\end{equation*}

Consider now the following analogue of the cyclotomic polynomials
$\Phi_n(x)$:
\begin{equation*}
\Phireg_n(x)= \prod_{k\in \Reg_n} \left(x-\exp(2\pi ik/n) \right).
\end{equation*}

The application of Proposition \ref{prop_general} gives the
following result.\\[2 pt]

\noindent {\bf Proposition 5.7.1}
\textit{For every $n\in \N$,}
\begin{equation*}
\Phireg_n(x)=\prod_{d\ddiv n} \Phi_d(x).
\end{equation*}

Here the degree of $\Phireg_n(x)$ is $\varrho(n)$. If $n$ is
squarefree, then $\Phireg_n(x)= x^n-1$ and for example,
$\Phireg_{12}(x)= \Phi_1(x)\Phi_3(x)\Phi_4(x)\Phi_{12}(x)
=x^9-x^6+x^3-1$.

It is well known that for every $n\in \N$, $n\ge 2$,
\begin{equation} \label{formula_U}
U(n):=\prod_{\substack{k=1\\ \gcd(k,n)=1}}^n \sin
\left(\frac{k\pi}{n}\right) =\frac{\Phi_n(1)}{2^{\phi(n)}},
\end{equation}
where
\begin{equation*}
\Phi_n(1) = \begin{cases} p, & n=p^{\, \nu}, \nu \ge 1, \\
1, & \text{otherwise, i.e., if $\omega(n)\ge 2$},
\end{cases}
\end{equation*}
and for $n\ge 3$,
\begin{equation} \label{formula_V}
V(n):=\prod_{\substack{k=1\\ \gcd(k,n)=1}}^n \cos
\left(\frac{k\pi}{n}\right) =\frac{\Phi_n(-1)}{(-4)^{\phi(n)/2}},
\end{equation}
where
\begin{equation*}
\Phi_n(-1)= \begin{cases} 2, & n=2^{\, \nu},  \\
p, & n=2p^{\, \nu}, p>2 \text{ prime},  \nu \ge 1,\\ 1, &
\text{otherwise}.
\end{cases}
\end{equation*}

For every $n\in \N$, $\prod_{k\in \Reg_n} \sin(k\pi/n)=0$, since
$n\in \Reg_n$. This suggests to consider also the modified products
\begin{equation*}
\Uregmod(n):= \prod_{\substack{k=1\\ k \text{\rm \ regular (mod
$n$)}}}^{n-1} \sin \left(\frac{k\pi}{n}\right),
\end{equation*}
\begin{equation*}
\Vregmod(n):= \prod_{\substack{k=1\\ k \text{\rm \ regular (mod
$n$)}}}^{n-1} \cos \left(\frac{k\pi}{n}\right).
\end{equation*}

We show that $\Uregmod(n)$ is nonzero for every $n\ge 2$. More
precisely, define the modified polynomials
\begin{equation*}
\Phiregmod_n(x)= (x-1)^{-1} \Phireg_n(x) = \prod_{\substack{d\ddiv
n\\d>1}} \Phi_d(x).
\end{equation*}

Here, for example, $\Phiregmod_{12}(x)=\Phi_3(x)\Phi_4(x)
\Phi_{12}(x) =x^8+x^7+x^6+x^2+x+1$. All of the polynomials
$\Phiregmod_n(x)$ have symmetric coefficients. By arguments similar
to those leading to the formulas \eqref{formula_U} and
\eqref{formula_V} we obtain the following identities.\\[2 pt]

\noindent {\bf Proposition 5.7.2} \label{prop_sin_cos}
\textit{For every $n\in \N$, $n\ge 2$,}
\begin{equation*}
\Uregmod(n)= \frac{\Phiregmod_n(1)}{2^{\varrho(n)-1}}
=\frac{\kappa(n)}{2^{\varrho(n)-1}},
\end{equation*}
\textit{and for every $n\in \N$, $n\ge 3$ odd,}
\begin{equation*}
\Vregmod(n)= \frac{\Phiregmod_n(-1)}{(-4)^{(\varrho(n)-1)/2}} =
\left(-1/4 \right)^{(\varrho(n)-1)/2}.
\end{equation*}

Note that $\varrho(n)$ is odd for every $n\in \N$ odd.

\section{Maximal orders of certain functions}

Let $\sigma(n)$ be the sum of divisors of $n$ and let
$\psi(n)=n\prod_{p\mid n} (1+1/p)$ be the Dedekind function. The
following open problems were formulated in \cite{Apo}:  What are the
maximal orders of the functions $\varrho(n)\sigma(n)$ and
$\varrho(n)\psi(n)$?

The answer is the following:

\begin{proposition}
\begin{equation*}
\limsup_{n\to \infty} \frac{\varrho(n)\sigma(n)}{n^2\log \log n}=
\limsup_{n\to \infty} \frac{\varrho(n)\psi(n)}{n^2\log \log n}=
\frac{6}{\pi^2} e^{\gamma},
\end{equation*}
where $\gamma$ is the Euler-Mascheroni constant.
\end{proposition}

\begin{proof} Apply the following general result, see \cite[Cor.\ 1]{TotWir2003}:
If $f$ is a nonnegative real-valued multiplicative
arithmetic function such that for each prime $p$,

i) $\rho(p):=\sup_{\nu \ge 0} f(p^{\nu})\le (1-1/p)^{-1}$, and

ii) there is an exponent $e_p=p^{o(1)}\in \N$ satisfying
$f(p^{e_p})\ge 1+1/p$,

then
\begin{equation*}
 \limsup_{n\to \infty} \frac{f(n)}{\log \log
n}=e^{\gamma}\prod_p \left(1-\frac1{p}\right) \rho(p).
\end{equation*}

Take $f(n)=\varrho(n)\sigma(n)/n^2$. Here $f(p)=1+1/p$ and 
$f(p^{\nu})= 1+1/p^{\nu}+1/p^{\nu+2}+1/p^{\nu+3}+ \ldots+1/p^{2\nu} < 1+1/p$ for
every prime $p$ and every $\nu \ge 2$.
This shows that $\rho(p)=1+1/p$ and obtain that
\begin{equation*}
\limsup_{n\to \infty} \frac{f(n)}{\log \log n}= e^{\gamma}\prod_p
\left(1-\frac1{p^2}\right) = \frac{6}{\pi^2} e^{\gamma}.
\end{equation*}

The proof is similar for the function $g(n)=\varrho(n)\psi(n)/n^2$.
In fact, $g(p)=f(p)=1+1/p$ and $g(p^{\nu})\le f(p^{\nu})$ for every
prime $p$ and every $\nu \ge 2$, therefore the result for $g(n)$
follows from the previous one.
\end{proof}

\begin{remark} {\rm Let $\sigma_s(n)=\sum_{d\div n} d^s$. Then for every real $s>1$,
\begin{equation*}
\limsup_{n\to \infty} \frac{\varrho_s(n)\sigma_s(n)}{n^{2s}} = \frac{\zeta(s)}{\zeta(2s)}.
\end{equation*}

This follows by observing that for $f_s(n)=\varrho_s(n)\sigma_s(n)/n^{2s}$,
$f_s(p)=1+1/p^s$ and $f_s(p^{\nu})= 1+1/p^{s\nu}+1/p^{s(\nu+2)}+
1/p^{s(\nu+3)}+\ldots+1/p^{2s\nu} < 1+1/p^s$ for every prime $p$ and
every $\nu \ge 2$. Hence, for every $n\in \N$,
\begin{equation*}
f_s(n) \le \prod_{p\div n} \left(1+\frac1{p^s}\right)< \prod_{p}
\left(1+\frac1{p^s}\right) = \frac{\zeta(s)}{\zeta(2s)},
\end{equation*}
and the $\limsup$ is attained for $n=n_k=\prod_{1\le j\le k} p_j$
with $k\to \infty$, where $p_j$ is the $j$-th prime.
}
\end{remark}

\section{Acknowledgement}
L. T\'oth gratefully acknowledges support from the Austrian Science
Fund (FWF) under the project Nr. M1376-N18.


\medskip

\noindent Br\u{a}du\c{t} Apostol \\ Pedagogic High School "Spiru
Haret" \\ Str. Timotei Cipariu, RO-620004 Foc\c{s}ani, Romania \\
E-mail: apo\_brad@yahoo.com

\medskip

\noindent L\'aszl\'o T\'oth \\
Institute of Mathematics, Universit\"at f\"ur Bodenkultur \\
Gregor Mendel-Stra{\ss}e 33, A-1180 Vienna, Austria \\ and \\
Department of Mathematics, University of P\'ecs \\ Ifj\'us\'ag u. 6,
H-7624 P\'ecs, Hungary \\ E-mail: ltoth@gamma.ttk.pte.hu

\end{document}